\documentclass[12pt,reqno]{article}

\usepackage[usenames]{color}
\usepackage{amssymb}
\usepackage{graphicx}
\usepackage{amscd} 

\usepackage[colorlinks=true,
linkcolor=webgreen,
filecolor=webbrown,
citecolor=webgreen]{hyperref}

\definecolor{webgreen}{rgb}{0,.5,0}
\definecolor{webbrown}{rgb}{.6,0,0}

\usepackage{color}
\usepackage{fullpage}
\usepackage{float}

\usepackage{graphics,amsmath,amssymb}
\usepackage{amsthm}
\usepackage{amsfonts}
\usepackage{latexsym}
\usepackage{epsf}

\usepackage{verbatim}

\usepackage{tikz}
\usepackage{forest}

\usepackage{mathtools}

\theoremstyle{plain}
\newtheorem{theorem}{Theorem}
\newtheorem{corollary}[theorem]{Corollary}
\newtheorem{lemma}[theorem]{Lemma}

\theoremstyle{definition}

\newtheorem{example}[theorem]{Example}

\theoremstyle{remark}

\begin{document}

\title{The No-Flippancy Game}
\author{
Isha Agarwal \and
Matvey Borodin \and
Aidan Duncan \and
Kaylee Ji \and
Shane Lee \and
Boyan Litchev \and
Anshul Rastogi \and
Garima Rastogi \and
Andrew Zhao \\
PRIMES STEP\\
\\
Tanya Khovanova\\
MIT
}

\maketitle

\begin{abstract}
We analyze a coin-based game with two players where, before starting the game, each player selects a string of length $n$ comprised of coin tosses. They alternate turns, choosing the outcome of a coin toss according to specific rules. As a result, the game is deterministic. The player whose string appears first wins. If neither player's string occurs, then the game must be infinite.

We study several aspects of this game. We show that if, after $4n-4$ turns, the game fails to cease, it must be infinite. Furthermore, we examine how a player may select their string to force a desired outcome. Finally, we describe the result of the game for particular cases.
\end{abstract}

\section{Introduction}
\label{introduction}

In this paper, we extensively analyze the \textit{No-Flippancy game}. It is inspired by the well-known Penney's coin game. Penney's game is defined as follows:

\begin{quote}
\textit{Penney's game.} Two players, Alice and Bob, individually select separate strings of a fixed length $n$. These strings are comprised of coin flip outcomes: H (denoting heads) and T (denoting tails). They toss a fair coin repeatedly until one player's selected string appears in the sequence of tosses. This player is then declared the winner.
\end{quote}

Penney's game produces interesting results and is thoroughly studied and discussed in \cite{BCG,C,F,MGSA,MG,GO,HN,WP,RV}.

We devised the No-Flippancy game to create a purely deterministic game akin to Penney's but independent of probabilities. The game is played by the following rules:

\begin{quote}
\textit{No-Flippancy Game.} As is in Penney's game, Alice and Bob choose distinct strings of length $n$ consisting of the letters H (for heads) and T (for tails). The two players alternate selecting the outcome of the next ``flip'' to add to the sequence by the rule delineated below:

	\textbf{The ``flip'' rule:} Let $i < n$ be the maximal length of a suffix of the sequence of chosen outcomes that coincides with a prefix of the current player's string. The player then selects the element of their string with index $i+1$ as the next term in the sequence. 

	Alice goes first, and whoever's string appears first in the sequence of choices wins.
\end{quote}

We demonstrate this game with an example.

Suppose $n=4$. Let Alice's chosen string be HHTT, and Bob's chosen string be THHH.

Alice begins by selecting the first term of the game outcome to be H. Here, $i=0$ for Bob, so he selects T as the next term. Our game outcome thus far is HT. As a result of his move, $i=0$ for Alice now, so she assigns the next flip to be H. The game outcome develops into HTH. For Bob, $i=2$ now. As a result, he takes the next term to be H. This makes the game outcome HTHH. Continuing as such, the following sequence of flip outcomes is produced:

\[ \textrm{HTHHTHHH.}\]

The game ceases here as the string THHH of one of the players, Bob, appears. Hence, he wins. 

This game and its many intriguing characteristics are the focus of this paper. In Section~\ref{sec:defs}, we introduce definitions, notations, and assumptions. In Section~\ref{sec:ig}, we illustrate that, given the string of length $n$, if the game lasts for over $4n-4$ turns, it must be infinite.

In Section~\ref{sec:fo}, we study how a player, when provided with the opponent's chosen string, may select a string to force a particular outcome. Furthermore, we demonstrate that, with a few exceptions, each player can force a win or an infinite game. Forcing a loss, which we discuss as well, is more complicated.

In Section~\ref{sec:whowins}, we discuss cases where we can determine the victor solely from the chosen strings of the two players. We describe numerous cases that result in an infinite game as well as cases where Alice's and Bob's strings share a common substring of lengths $n-1$ and $n-2$. As we conclude this section, we explore particular instances where one player has a string consisting of either entirely the same letter or alternating Hs and Ts.

In Section~\ref{sec:cd}, we display data that we computed. For $n < 9$, we calculate the maximal possible number of turns in a finite game. We also compute the number of string pairs that Alice and Bob can choose to attain a particular outcome: Alice wins, Bob wins, or the game is infinite.

\section{Definitions}\label{sec:defs}

First, we denote Alice's string by $A$, where $A = A_{1}A_{2}A_{3} \cdots A_{n}$. Each $A_{i}$ is either H (representing a \textit{heads} outcome) or T (representing a \textit{tails} outcome). Similarly, Bob's string is denoted by $B$, where $B = B_{1}B_{2}B_{3} \cdots B_{n}$, with each $B_{i}$ being either H or T. In this game, Alice and Bob have strings of the same length $n$.

We say that two strings are \textit{complementary} to one another if, for each letter H in a given position in one string, there is a T in the same position in the other string and vice versa. We denote the complement of a string by using an overline: $\overline{\textrm{H}} = \textrm{T}$ and $\overline{\textrm{T}} = \textrm{H}$. For example, $\overline{\textrm{HHT}} = \textrm{TTH}$.

Without loss of generality, we may assume that Alice's string begins with H.

We say that a string $A = A_{1}A_{2}A_{3} \cdots A_{n}$ \textit{alternates} if $A_i \neq A_{i+1}$, where $1 \leq i < n$.

Let us denote the sequence of letters that represents the game output after $k$ tosses as $G^k$, where $G^k = G_{1}G_{2}G_{3} \cdots G_{k}$, with each $G_{i}$ being either H or T.

We define Alice's progress as the length of the largest suffix in $G^k$ that is a prefix of Alice's string. Formally, Alice's progress is the largest $i$ such that $A_{1}A_{2}A_{3} \cdots A_{i} = G_{k-i+1}G_{k-i+2}\cdots G_{k-1}G_{k}$.  We denote Alice's progress at the $k$-th step of the game as $a_k$. We define Bob's progress similarly and denote it as $b_k$. If we are referring to one of the players, Alice or Bob, we denote this player's progress as $p_k$. By the definition of the game, on player $P$'s turn, player $P$ selects a toss such that $p_{k+1} = p_k +1$. 

The continuation of the game is entirely defined by the pair of numbers $(a_k,b_k)$ and whose turn it is. We call the triplet $(a_k,b_k,P)$ \textit{the state of the game}. Here, the element $P$ represents the player with the current turn. The game always starts as $(0,0,A)$. After the first turn the state of the game might be either $(1,0,B)$ if Alice and Bob have different first characters in their strings or $(1,0,B)$ if the characters are the same.

The states of the example game above are as follows:

\[(0,0,A), (1,0,B), (0,1,A), (1,2,B), (2,3,A), (3,1,B), (1,2,A), (2,3,B), (3,4,A).\]

Since $b_k$ reaches $n$ before $a_k$, Bob wins.

If it is Alice's turn, then the state is $(a_k,b_k,A)$, and at the next turn it is $(a_{k+1},b_{k+1},B)$, where $a_{k+1} = a_k +1$, and $b_{k+1} \leq b_k + 1$. Bob's turn changes the state in a similar fashion.  

We call the number $\max\{a_k,b_k\}$ the \textit{progress} of the game after $k$ tosses. The game ends in a win when the progress reaches $n$.

We call a move \textit{synchronized} when both players want the same toss. In other words, when the state of the game changes from $(a_k,b_k,P_1)$ to $(a_k+1,b_k+1,P_2)$, the move is synchronized.

\section{The infinite game}\label{sec:ig}

It is possible for the game to extend infinitely, with neither player's string ever occurring; for example, if Alice's string is HH and Bob's string is TT, then the game output becomes an infinite alternating string HTHTHT$\ldots$, and no one wins.

When utilizing a computer to calculate the progress of a game, it is useful to know when to stop the calculation. The infinite games must be periodic, but at what point do we know that a game must be infinite?

It is not enough to observe that the beginning of the output is periodic. For example, if Alice's string is HHHHHHHH and Bob's is HTHTHTHT, the output string commences periodically as HTHTHT but is not an infinite string.

The following lemma allows one to recognize an infinite game.

\begin{lemma}
If the game takes longer than $4n-4$ turns, where $n$ is the length of Alice's and Bob's strings, then it must be infinite.
\end{lemma}

\begin{proof}
Suppose, after $k$ tosses, Alice is not winning; that is, $a_k \leq b_k$. In this case, the value of $a_k$ is uniquely defined. Indeed, if one were to take the prefix of Bob's string of length $b_k$, then $a_k$ is the largest suffix of this string that is a prefix of Alice's string. This means that for any maximum progress $m$, there are no more than 4 different states of the game with this maximum: the maximum and the move could belong to either Alice or Bob.

Additionally, after Alice's turn, $a_k$ cannot be 0. The same is true for Bob. This means that the following states can never occur: $(0,0,B)$, $(0,1,B)$, and $(1,0,A)$. Thus, there is only one state of the game with a maximum of 0: $(0,0,A)$.

Let us consider a maximum progress of 1. If the game contains state $(1,1,P)$, then both Alice's and Bob's string start with the same letter. This means that if the game contains state $(1,1,P)$, it cannot contain states $(0,1,P)$ or $(1,0,P)$. Therefore, there are no more than two states in the game with a maximum progress of 1.

For the maximum progress $n$, the game can only contain one such state as it immediately ends.

To recapitulate, we have not more than one state with the maximum progress of 0, not more than two states with the maximum progress of 1, not more than one state with the maximum progress of $n$, and not more than four states for the maximum progress of $i$, where $1 < i < n$. The total bound is $1+2+1 + 4(n-2) = 4n-4$.

The state of the game uniquely defines the rest of the game. This means that if the state of the game repeats at any point, the game must be infinite. As a result, a finite game cannot have more than $4n-4$ turns.
\end{proof}

One might wonder how to calculate the precise maximum length of a finite game. In Section ~\ref{sec:cd}, we present data for the maximum length of finite games for values of $n$ less than 9. They are all noticeably less than $4n-4$. In the proof of Theorem~\ref{thm:Aliceforcesloss}, we see examples of games that take $n$ or $n+1$ turns. Additionally, the proof contains an example of a game that takes $n+x$ turns, where $n$ and $x$ are odd and $x \leq n-2$. Thus, there are games with up to $2n-2$ turns.

\section{Forcing the outcome}\label{sec:fo}

Suppose one player announces their string first, and the other player wishes to obtain a particular outcome: a win, loss, or tie. In this section, we discuss how the other player may construct their string to attain their desired game outcome.

\subsection{Forcing an infinite game}

We demonstrate that each player can force an infinite game, except in specific cases where the strings are short and alternate.

\subsubsection{Bob forces an infinite game}

If Bob wants an infinite game, he can force it with a few exceptions. 

\begin{lemma}
Assume that Alice's string starts with H, and Bob chooses his string of the same length afterward. Bob can force an infinite game for any string Alice chooses except H, HT, HTH, or HTHT. 
\end{lemma}

\begin{proof}
First, we show that an infinite game is impossible if Alice chooses H, HT, HTH, or HTHT. 
\begin{itemize}
\item H: If Alice chooses H, she wins immediately, as Bob is not allowed to pick the same string as Alice.
\item HT: For the first turn, Alice picks H. For the second turn, Bob either picks T and wins for Alice or picks H. If Bob does not win on the second turn, Alice wins on her next turn.
\item HTH: Alice chooses H on her first move. If, after the first two moves, the output is HT, then Alice wins on the next move. If the output after two moves is HH, then Bob's string must start with HH. On the third move, Alice picks T. If Bob's string is HHT he wins. Otherwise, his string is HHH and he has to pick H. In this case, Alice wins.
\item HTHT: Alice starts with H. We consider what happens next  depending on the prefix of Bob's string. 
\begin{itemize}
\item HH: Suppose Bob's string starts with HH. Then Bob chooses H on the second turn, and Alice chooses T afterward. If Bob's string started with HHT, he wins the next move. Otherwise, his string starts with HHH. In this case, Bob picks H. Then Alice picks T and wins, so the game does not result in an infinite game.
\item HT: First, Alice chooses H. Bob chooses T. Alice chooses H. If Bob's string starts with HTH, he wins on the next turn. Otherwise, his string starts with HTT, and he picks T, so Alice wins, and the game does not result in an infinite game.
\item T: First, Alice chooses H. Bob chooses T, and then Alice picks H. If Bob's string starts with TT or THT, he picks T next, and Alice wins. If it starts with THH, he chooses H next. Alice then picks T, and the output now is HTHHT. If Bob's string is THHT, he wins. Otherwise, his string is THHH. In this case, Bob chooses H, and Alice chooses T. Alice wins, and the game does not result in an infinite game.
\end{itemize}
\end{itemize}

Now, we show that for other strings, Bob can force an infinite game.

Suppose Alice chooses a string with two consecutive terms that are the same. Without loss of generality, we can assume that these terms are HH. Then, Bob can choose a string comprised of only Ts. This way, Bob always chooses a T, and there is no way for Alice to have two Hs in a row. Additionally, Bob cannot win as Alice is guaranteed to pick an H before Bob obtains his string.

Alternatively, Alice could choose a string of alternating Hs and Ts with $n \geq 5$. Without loss of generality, we may assume that her string starts as  HTHTH. Bob can choose any string beginning with HHTTT. The game then proceeds periodically and infinitely as HHTTHHTT$\ldots$.
\end{proof}

\subsubsection{Alice forces an infinite game}

If Alice wishes for an infinite game, she can force it with a small number of exceptions. Without a loss of generality, we can assume that Bob's string starts with H.

\begin{lemma}
If Bob's string starts with H, then, the only possible strings for Bob that always end in a finite game are H, HT, HTH, HTHT, HTHTH.
\end{lemma}

\begin{proof}
Similar to the previous section, we can perform an exhaustive search to show that there are no infinite games for Bob's choice of H, HT, HTH, HTHT, or HTHTH. 

Suppose Bob chooses a string with two consecutive terms that are the same. Without loss of generality, we can assume that these terms are HH. Then, Alice can make her string all Ts. This way, Alice always chooses a T, so Bob cannot get two Hs in a row. Additionally, Alice cannot win as Bob always picks an H before Alice finishes her string.

Suppose Bob chooses a string of alternating Hs and Ts with $n \geq 6$. Without loss of generality, we can assume that his string starts as  HTHTHT. Alice can choose any string starting with THHTTT. Then, the game proceeds periodically and infinitely as THHTTHHTT$\ldots$.
\end{proof}

\subsection{Forcing a win}

We show that each player can force a win with one exception: if $n=1$, then Alice always wins. It is possible to force a win in less than $n+2$ turns.

\subsubsection{Bob forces a win for himself}

For $n=1$, Alice will always win.

\begin{lemma}
If $n > 1$ and Bob chooses his string after Alice, Bob can force a win for himself. Moreover, he can win the game in less than $n+2$ turns.
\end{lemma}

\begin{proof}
We have two cases, depending on whether Alice's string alternates or not. 

Suppose Alice's string alternates. Here, Bob can choose $B = A_{1}A_{1}A_{2}A_{3}$ $\cdots$ $A_{n-1}$. After the first two turns, the game outcome is $A_{1}A_{1}$. From here, the moves are synchronized and $b_k = a_k +1$ consistently. Thus, Bob wins on the $n$-th move. 

Suppose Alice's string does not alternate. Consider the minimum $k$ such that $A_{k} = A_{k+1}$. Then, Bob chooses his string as $B = \overline{A_{k}}A_{1}A_{2}A_{3}$ $\cdots$ $A_{n-2}A_{n-1}$. Without loss of generality, we assume that $A_1 = H$. It is worth noting that if $k = 2$, then Bob's string is also the best choice for Bob in the Penney's game \cite{F}.

Suppose $k$ is odd. Then $A_k = $H, and both Bob's and Alice's strings alternate for the first $k$ characters, because $A_k$ is defined as the \textbf{first} occurrence of two letters in a row. The output of the game is the alternating string for the first $k$ turns, so the state of the game is $(k,k-1,B)$ Then, Bob chooses T on turn $k+1$ as he is behind Alice in his alternating prefix, making the state of the game $(k-1,k,A)$. After that, the moves are synchronized with $b_k = a_k +1$. Bob wins on turn $n+1$.

Suppose $k$ is even. Since $A_k = T$ and $A_1 = H$, Bob's string starts as HH. Alice picks H on her first move, and Bob selects H afterward. Following this, the moves are synchronized with Bob being one character ahead of Alice. Bob wins on move $n$.
\end{proof}

Below are some examples.

\begin{example} Suppose $A$ = HTHT. Alice's string alternates. In accordance with rules above, Bob chooses HHTH. The game output is HHTH.
\end{example}

\begin{example} Suppose $A$ = HTHHT. In this case, $k = 3$, since $A_{3} = A_{4} = H$. Thus, Bob's string is THTHH. The game proceeds as HTHTHH and Bob wins.
\end{example}

\begin{example} Suppose $A$ = HTTH. In this case, $k = 2$, since $A_{2} = A_{3} = T$. Thus, Bob's string is HHTT. The game proceeds as HHTT and Bob wins.
\end{example}

\subsubsection{Alice forces a win for herself}

Next, we discuss when Alice can force a win. 

\begin{lemma}
If Alice chooses her string after Bob, she can force a win for herself. Moreover, she can win in $n$ moves.
\end{lemma}

\begin{proof}
Suppose Bob's string is $B = B_{1}B_{2}B_{3}$ $\cdots$ $B_{n}$. Then, Alice can choose $A = \overline{B_{1}}B_{1}B_{2}B_{3}$ $\cdots$ $B_{n-1}$. Starting from the second move, each move is synchronized, with Alice being one character ahead. Alice wins on move $n$.
\end{proof}

\subsection{Forcing a loss}

Forcing a loss is more complicated than forcing a win or an infinite game. We have acquired some results in this section.

\subsubsection{Alice forces a loss for herself}

The following lemma describes the case where Alice cannot force a loss for herself.

\begin{lemma}
If $n$ is odd and all the characters in Bob's string are the same, Alice cannot force a loss.
\end{lemma}

\begin{proof}
If $n=1$, Alice always wins. Suppose $n >1$. Without loss of generality, suppose Bob has all Ts. Because Alice and Bob have different strings, Alice cannot have $n$ Ts in a row. Thus, Bob cannot win on Alice's turn. 

Suppose Bob wins on his turn. For Bob to have won, Alice must have placed a T before Bob's winning turn, meaning that the game output has $n-1$ Ts in a row. This implies that Alice's string has $n-1$ Ts in a row. We can manually check the two cases: Alice's string must either start or end with H. In both cases, she wins.
\end{proof}

Our preliminary computation shows that, disregarding the case above, Alice can force a loss. However, we were unable to prove this. The following theorem presents all the cases when we can prove that Alice can force a loss.

\begin{theorem}\label{thm:Aliceforcesloss}
Alice can force a loss in the following cases:
\begin{enumerate}
\item If $n$ is even.
\item If Bob's string starts with HT.
\item If Bob's string starts with an even number of Hs.
\item If Bob's string starts with an odd number of Hs followed by at least two Ts.
\end{enumerate}
\end{theorem}

\begin{proof}
We present a proof separately for each case.
\begin{enumerate}
\item Alice can choose a string such that everything but the last character is the same as Bob's string i.e. $A = B_1B_2B_3 \ldots B_{n-2}B_{n-1} \overline{B_{n}}$.
When this occurs, Bob wins because the first $n-1$ turns are synchronized, and for the $n$-th turn, Bob chooses his last outcome. 
\item If Bob's string starts with HT, then the start of his string is alternating. As a result, the string must either cease alternating at some point or continue alternating until the string ends, giving us the following two cases:

Case 1: Bob's entire string is alternating. Alice can choose a string with all Ts to guarantee that Bob wins.

Case 2: The prefix of Bob's string which alternates ends at some point. For the alternating to stop, two of the same letter must conclude the alternating prefix, resulting in two more cases. Suppose that $k$ is the position of the second repeated letter right after the alternating prefix in Bob's string. 
\begin{enumerate}
    \item Subcase 1: The string which alternates ceases with two Hs. In this case, $k$ must be even. Therefore, $B=$ HTHTHT...THH$B_{k+1}B_{k+2}...B_n$. Note that there are $\dfrac{k}{2}-1$ pairs of HTs in the alternating string. If Alice copies Bob's string but removes the first two tosses and adds two of her own at the end, then $A=$ HTHT...THH$B_{k+1}B_{k+2}...B_{n}A_{n-1}A_{n}$, she can guarantee that Bob wins. This is because she has $\dfrac{k}{2} - 2$ pairs of HTs, which is one less than Bob's. The game output starts with $\dfrac{k}{2}-2$ pairs of HTs, so the state of the game is $(k-4,k-4,B)$. Alice then chooses an H, and Bob chooses a T. At this point, the state of the game becomes $(k-4,k-2,B)$. After this, moves are synchronized and Bob stays ahead by 2 tosses until he wins. Bob wins in $n$ moves.
    \item Subcase 2: The string which alternates ends with two Ts. In this case, $k$ must be odd.  We can represent Bob's string as HTHT...HTT$B_{k+1}B_{k+2}...B_n$. If Alice copies Bob's string but removes the first character and adds another character of her own at the end, meaning that her string would be THT...HTT$B_{k+1}B_{k+2}...B_nA_{n}$, she can guarantee that Bob wins. In the beginning, the game output alternates and Alice is ahead by 1 toss. On move $k-1$ Bob chooses H, when Alice needs T. Now Alice is set back by 2 tosses, and the game continues synchronized until Bob wins. Bob wins in $n+1$ moves.
\end{enumerate}

\item When Bob's string starts with an even number of Hs, Alice can cut off the first character of Bob's string, so $A = B_2B_3B_4 \ldots B_{n-1}B_nA_{n}$. In the beginning, the game is synchronized. Then Bob chooses H when Alice needs T, setting Alice back by one toss. The game proceeds synchronized until Bob wins on turn $n$.
\item Suppose $n$ is odd. Let us suppose that the first $x$ terms of $B$ are all Hs and $B_{x+1} = B_{x+2} = \textrm{T}$, where $x$ is odd and greater than $1$. In particular, we have $x+2 \leq n$.

Now Alice chooses her string so that her prefix of size $n-1$ equals Bob's suffix of the same length. That is $A_i = B_{i+1}$, for $i < n$, so $A = B_2B_3B_4 \ldots B_{n-1}B_nA_{n}$.

For the first $x-1$ turns, both $A$ and $B$ are synchronized. As $x$ is odd, the $x$-th turn must belong to Alice, who must place $A_{x} = B_{x+1} = \overline{B_{1}} =\textrm{T}$. Bob dissents and is forced to restart his string by placing $B_{1} = \textrm{H}$ again. However, $A_{x+1} = B_{x+2} = \overline{B_{1}} = \textrm{T}$, so Alice is forced to restart as well. Thus, they both restart the whole game on Bob's turn. After this restart, Bob is able to get the $x$-th term and places $B_{x}$. From here, the game is synchronized and Bob is one character ahead of Alice. Bob wins on move $x+n$.
\end{enumerate}
\end{proof}

\subsubsection{Bob forces a loss for himself}

When there is an odd number of tosses, Bob can make a string such that everything but the last toss is the same as Alice's string. When this happens, Alice wins because she always gets the last coin.

Bob cannot always force a loss when there is an even number of coins. 

\begin{lemma}
If $n$ is even and all characters in Alice's string are the same, Bob cannot force a loss.
\end{lemma}

\begin{proof}
Without loss of generality, suppose Alice has all Ts. As Bob has a different string, he does not have $n$ Ts in a row. Thus, Alice cannot win on Bob's turn. Suppose Alice wins on her turn. That means the previous turn was Bob's and he placed a T, meaning the game output has $n-1$ Ts in a row. This implies that Bob's string has $n-1$ Ts in a row. We can manually check the two cases: 1) Bob's string is HTTT$\ldots$ or 2)  $B =$ TTT$\ldots$ TTH. 
\begin{enumerate}
    \item In this case, Alice first places a T, then Bob places an H. After this, Bob is one move ahead of Alice and their moves are synchronized for the rest of the game. Thus, Bob wins in $n+1$ moves.
    \item In this case, Alice and Bob are synchronized for the first $n-1$ moves. Since $n$ is even, Bob has the $n$-th turn and places an H, securing himself a victory. 
\end{enumerate}

In both cases, Bob wins.
\end{proof}

We wrote a program and found some strings for which Bob cannot force a loss. We looked only at strings that start with H. We excluded the strings where all the characters are the same. We got one case for length 4: HHTT, three cases for length 6: HHTTTT, HHTHTT, HHHHTT, and 9 cases for length 8: HHTTTTTT, HHTTTHTT, HHTTHTTT, HHTHTTTT, HHTHTHTT, HHTHHTTT, HHHHTTTT, HHHHTHTT, HHHHHHTT.

One can notice that each string starts with an even number of Hs. We can prove that, if this is not the case, Bob can force a loss.

\begin{lemma}
If Alice's string starts with an odd number of Hs, Bob can force a loss for himself.
\end{lemma}

\begin{proof}
Bob chooses a string such that his prefix of length $n-1$ is the same as Alice's suffix, so $B = A_2A_3A_4...A_nB_n$. Suppose Alice has a run of $k$ Hs in the beginning, where $k$ is odd. Then, the game output starts with $k$ Hs. At this point, the state of the game is $(k,k-1,B)$. After this, all the moves are synchronized and Alice is one character ahead. Alice wins in $n$ moves.
\end{proof}

\section{Cases where we can determine who wins}\label{sec:whowins}

In this section, we study several cases in which we can determine the winner. We start with cases that end in an infinite game. Next, we examine pairs of strings with a large overlap. Finally, we discuss particular cases when one of the strings either alternates or consists of the same letter.

\subsection{An infinite game}

We denote the length of the largest substring of consecutive Hs or Ts in Alice's prefix of length $p$ as $h_A(p)$ or $t_A(p)$ correspondingly. We use a similar notation for Bob.

\begin{lemma}\label{lemma:longestrun}
If $h_A(n)+2 \leq h_B(n)$ or $t_A(n)+2 \leq t_B(n)$, then Bob cannot win. The same is true if we interchange Alice and Bob.
\end{lemma}

\begin{proof}
By symmetry, it is enough to prove this lemma for the case $h_A(n)+2 \leq h_B(n)$.
If Alice sees a game output with $h_A(n)$ or $h_A(n)+1$ occurrences of letter H at the end, then she will interrupt it by placing a T. Thus, Bob can never finish his largest run of Hs.
\end{proof}

The above lemma allows us to describe some infinite games.

\begin{theorem}\label{thm:infinitegame}
If there exists $p$ such that
\[h_A(p) +1 < h_B(p) \quad \textrm{ and } \quad t_B(p) + 1 < t_A(p)\]
or
\[h_B(p) +1 < h_A(p) \quad \textrm{ and } \quad t_A(p) + 1 < t_B(p),\]
then the game is infinite.
\end{theorem}

\begin{proof}
Applying Lemma~\ref{lemma:longestrun}, to $p$, neither Alice nor Bob can achieve their prefix, and thus neither can win. Thus, the game must be infinite.
\end{proof}

\begin{corollary}\label{cor:HHvsTT}
If Alice's string starts with HH and Bob's string starts with TT or vice versa, then the game is infinite.
\end{corollary}

\begin{proof}
The second condition of Theorem~\ref{thm:infinitegame} is satisfied for $p=2$.
\end{proof}

\subsection{Cases with large overlaps}

Here, we discuss the result of the game when Alice's and Bob's strings have a common substring with a length of at least $n-2$.

\begin{lemma}\label{lemma:largecommonprefix}
Suppose Alice and Bob have the same string except for the last letter. If $n$ is even Bob wins, and if $n$ is odd Alice wins.
\end{lemma}

\begin{proof}
We have $A = B_1B_2B_3 \ldots B_{n-2}B_{n-1}\overline{B_{n}}$. The first $n-1$ tosses are synchronized since they coincide with the common prefix of Alice and Bob. After $n-1$ turns, the person whose turn it is finishes his/her string and wins. If $n$ is even, this person will be Bob, and conversely, if $n$ is odd, it will be Alice.
\end{proof}

Now, we consider cases where one player's suffix equals the other player's prefix both of length at least $n-2$.

\begin{lemma}
If Alice's string starts with HT, Bob's with HH, and Bob's suffix of length $n-1$ is the same as Alice's prefix of length $n-1$, then Bob wins.
\end{lemma}

\begin{proof}
We have $B =$ HHT$A_3A_4A_5 \ldots A_{n-2}A_{n-1}$. In this case, Alice chooses H on her first turn. Then, Bob chooses another H. From then on, the game is synchronized with Bob being one character ahead. Thus, Bob wins.
\end{proof}

\begin{lemma}
If Alice's string starts with HH, Bob's with HT, and Bob's suffix of length $n-2$ is the same as Alice's prefix of length $n-2$, then Bob wins.
\end{lemma}

\begin{proof}
We have $B =$ HTHH$A_3A_4A_5 \ldots A_{n-3}A_{n-2}$. In this case, the game output starts as HT. Alice then has to restart at the third toss of the game, while Bob does not have to restart. From here, the game is synchronized, and Bob is two letters ahead. Therefore, Bob wins.
\end{proof}

\subsection{Particular strings}

Now, we look at cases where Alice or Bob have either a string of all the same letters or an alternating string. We completely resolve the case when one player has a string of all the same letters.

\begin{lemma}
If Alice's string consists of solely Hs, then the only way for her to win is if Bob has HHH$\ldots$HHT, and $n$ is odd. Otherwise, Bob wins if and only if he has Hs in either all even or all odd spots i.e. $B =$ H$B_2$H$B_4$H$B_6 \ldots$ or $B=B_1$H$B_3$H$B_5$H $\ldots$. In all other cases, the game is infinite.
\end{lemma}

\begin{proof}
Alice cannot win on Bob's turn. Indeed, if Bob puts down the last H,  which means his string contains $n$ Hs, which is impossible. For Alice to win, Bob needs to place the second to last H. This means Bob's string contains $n-1$ Hs. If  $B_1= $ T and $n > 1$, then Bob wins. This is because after Bob chooses T, Bob and Alice are synchronized, and since Bob requires fewer H's he gets his string first. Thus, Alice wins only if Bob has the same $(n-1)$-prefix as Alice, and by Lemma~\ref{lemma:largecommonprefix} this means that $n$ must be odd.

I f Bob has all Hs in odd positions, then Alice will supply all the Hs that Bob needs. Each game output equals Bob's prefix and Bob wins in $n$ moves. On the other hand, if Bob has Hs in every even position, we can insert an imaginary H into the start of his string, and use the same logic. Thus, ignoring the first H, the game proceeds with Bob's string. The game will end in $n+1$ moves.

Since Alice's string is made up of only Hs, she will always pick H as her choice. Thus, the only way for Bob to win is if he has Hs either in all even or all odd positions. If this is not the case, then the game is infinite as neither player can achieve their string.
\end{proof}

\begin{lemma}
If Bob has HHH...HHH, the only way for him to win is if Alice has the same string except for the last letter, and $n$ is even. Otherwise, Alice wins if and only if Alice has Hs in either all even or all odd spots. In all other cases, the game is infinite.
\end{lemma}

\begin{proof}
With minor adjustments, the proof is the same as in the previous lemma. If Alice has all Hs in even positions, the game ends in $n$ moves. If she has all Hs in odd positions, the game ends in $n+1$ moves. 
\end{proof}

We have partial results for alternating strings.

\begin{lemma}
If one player has an alternating string starting with H: HTHTHT... and the other player has a string that starts with TT, then the person with the alternating string wins.
\end{lemma}

\begin{proof}
The player, whose string starts with TT, is unable to continue past this prefix due to the alternating string of the other player. Hence, they always choose T. The player with the alternating string always chooses H. If Alice has an alternating string, she wins in $n$ moves. If Bob has an alternating string, he wins in $n+1$ moves, since the first move made by Alice will not benefit him.
\end{proof}

\section{Computational data}\label{sec:cd}

We wrote several programs to analyze this game. Our results are in this section.

\subsection{The game length for finite games}

The longest game length for a finite game as a function of string length $n$ is provided by the sequence
\[1,\ 3,\ 4,\ 8,\ 9,\ 13,\ 18,\ 22,\ \ldots.\]

For example, suppose $n=2$. We show all possible combinations in Table~\ref{table:n2}, where we assume that Alice's string starts with H. The longest finite game happens if Alice chooses HH and Bob chooses TH.

\begin{table}[ht!]
\begin{center}
\begin{tabular}{|c|c|c|c|}
\hline
Alice & Bob & Output & Winner\\
\hline
HH & HT & HT & Bob\\
HH & TH & HTH & Bob\\
HH & TT & HTHTHT... & Infinite\\
HT & HH & HH & Bob\\
HT & TH & HT & Alice\\
HT & TT & HT & Alice\\
\hline
\end{tabular}
\end{center}
\caption{Different Outcomes for $n = 2$.}
\label{table:n2}
\end{table}

\subsection{Number of infinite games}

We assume that Bob's string is not equal to Alice's string. So the number of possible pairs of strings is $2^n (2^{n} - 1)$ for a given length $n$.

The following Table~\ref{table:ndo} provides a tally of the results.

\begin{table}[ht!]
\begin{center}
\begin{tabular}{cc|ccc}\\
Size & Total & Bob Wins & Alice Wins & Infinite \\\hline
1 &2&  0 & 2 & 0\\
2 &12 & 6 & 4 & 2\\
3 &56& 16 & 26 & 14\\
4 &240& 84 & 64 & 92\\
5 &992& 238  & 290  & 464\\
6 &4032& 916& 756 & 2360\\
7 &16256& 2636 & 2932 & 10688\\
8 &65280& 8942 & 7774 & 48564\\
\end{tabular}
\end{center}
\caption{Number of different outcomes.}
\label{table:ndo}
\end{table}

Table~\ref{table:pdo} shows a proportion of the outcomes. The results are approximate.

\begin{table}[ht!]
\begin{center}
\begin{tabular}{cc|ccc}\\
Size & Total & Bob Wins & Alice Wins & Infinite \\\hline
1 &2&  0 & 1 & 0\\
2 &12 & 0.5 & 0.3 & 0.17\\
3 &56& 0.29 & 0.46 & 0.25\\
4 &240& 0.35 & 0.27 & 0.38\\
5 &992& 0.24  & 0.29  & 0.47\\
6 &4032& 0.23 & 0.19 & 0.59\\
7 &16256& 0.16 & 0.18 & 0.66\\
8 &65280& 0.14 & 0.12 & 0.74\\
\end{tabular}
\end{center}
\caption{Proportion of different outcomes.}
\label{table:pdo}
\end{table}

\begin{lemma}\label{lemma:increaseinfportion}
The proportion of infinite games increases with the increase of the string length.
\end{lemma}

\begin{proof}
If the game is infinite for Alice's string $X$ and Bob's string $Y$ of the same length, then it is infinite for any longer string pairs where Alice's prefix is $X$ and Bob's prefix is $Y$.

Additionally, with each increasing value of $n$, we get new pairs of strings that result in an infinite game. Consider an example, where Alice has HHH$\dots$H and Bob has HHH$\dots$HHT, both of length $n$. By Lemma~\ref{lemma:largecommonprefix} one of the players wins in $n$ moves. For strings of length $n+1$ let Alice add letter H at the end of her string and let Bob add letter T at the end of his string. By Theorem~\ref{thm:infinitegame} the new game is infinite.
\end{proof}

From Corollary~\ref{cor:HHvsTT} we know that the number of string pairs that end in an infinite game is at least $2^{2n-3}$, which is at least 1/8 of all the pairs. Lemma~\ref{lemma:increaseinfportion} informs us that the portion of such games increases alongside $n$.

From Lemma~\ref{lemma:largecommonprefix} we know that the number of string pairs of length $n$ that end in someone's win is at least $2^n$.

\section{Future research}

Our examination of the No-Flippancy game yielded intriguing characteristics. The research can be continued as follows.

\begin{enumerate}
    \item We showed that if the game takes more than $4n-4$ turns it must be infinite. We also found games that take $2n-2$ turns. Our computational results show that the largest game length of a finite game lies somewhere in between. If would be interesting to know the true maximum length of a finite game. This question is related to a discussion of which game states are possible, and also which sequences of game states are possible.
    \item For all $n > 1$, Alice and Bob can force a win given the opponent's string. For all $n > 5$, they can also force an infinite game. Finally, in many cases, they can also force a loss. The case for forcing a loss was not completely analyzed. For example, one might notice that our computational results show that for all of Alice's strings for which Bob cannot force a loss, if they start with H and not all the characters are Hs, then they have at least two Ts at the end. It might be possible to explicitly describe all the strings for which Bob cannot force a loss.  
    \item With certain string assignments, we can determine who the winner must be without simulating the game. These partial results could be extended to more cases.
\end{enumerate}

There are yet a myriad of other directions in which research for the No-Flippancy game may be undertaken. We briefly list potential research avenues below:

\begin{enumerate}
    \item Alice and Bob do not have to alternate turns. They can pick turns randomly or follow some sequence. For example, one natural sequence for taking turns is the Thue-Morse sequence.
    \item Alice and Bob can use more outcomes than just H and T. For example, they can use a dice instead of a coin, or any set of letters.
	\item We can increase the number of players. For example, we can add Charlie to the company of Alice and Bob.
\end{enumerate}

We hope that other researchers are as fascinated with this game as we are and will continue studying it.

\section{Acknowledgments}

This project was done as part of MIT PRIMES STEP, a program that allows students in grades 7 through 9 to try research in mathematics. Tanya Khovanova is the mentor of this project. We are very grateful to PRIMES STEP for this opportunity.

\end{document}